\documentclass[reqno]{amsart}
\usepackage[T1]{fontenc}
\usepackage[margin=1in]{geometry}
\usepackage{amsmath}
\usepackage{amssymb}
\usepackage{amsthm}
\usepackage{multicol}
\begin{document}
\def\eq#1{{\rm(\ref{#1})}}
\theoremstyle{plain}
\newtheorem*{theo}{Theorem}
\newtheorem*{ack}{Acknowledgements}
\newtheorem*{pro}{Proposition}
\newtheorem*{coro}{Corollary}
\newtheorem*{lemm}{Lemma}
\newtheorem{thm}{Theorem}[section]
\newtheorem{lem}[thm]{Lemma}
\newtheorem{prop}[thm]{Proposition}
\newtheorem{cor}[thm]{Corollary}
\theoremstyle{definition}
\newtheorem{dfn}[thm]{Definition}
\newtheorem*{rem}{Remark}
\def\coker{\mathop{\rm coker}}
\def\codim{\mathop{\rm codim}}
\def\ind{\mathop{\rm ind}}
\def\Re{\mathop{\rm Re}}
\def\Vol{\rm Vol}
\def\Im{\mathop{\rm Im}}
\def\im{\mathop{\rm im}}
\def\sp{\mathop{\rm span}}
\def\Hol{{\textstyle\mathop{\rm Hol}}}
\def\C{{\mathbin{\mathbb C}}}
\def\R{{\mathbin{\mathbb R}}}
\def\N{{\mathbin{\mathbb N}}}
\def\Z{{\mathbin{\mathbb Z}}}
\def\O{{\mathbin{\mathbb O}}}
\def\L{{\mathbin{\mathcal L}}}
\def\X{{\mathbin{\mathcal X}}}
\def\CV{{\mathbin{\mathcal C}(V)}}
\def\al{\alpha}
\def\be{\beta}
\def\ga{\gamma}
\def\de{\delta}
\def\ep{\epsilon}
\def\io{\iota}
\def\ka{\kappa}
\def\la{\lambda}
\def\ze{\zeta}
\def\th{\theta}
\def\vt{\vartheta}
\def\vp{\varphi}
\def\si{\sigma}
\def\up{\upsilon}
\def\om{\omega}
\def\De{\Delta}
\def\Ga{\Gamma}
\def\Th{\Theta}
\def\La{\Lambda}
\def\Om{\Omega}
\def\Up{\Upsilon}
\def\sm{\setminus}
\def\na{\nabla}
\def\pd{\partial}
\def\op{\oplus}
\def\ot{\otimes}
\def\bigop{\bigoplus}
\def\bigot{\bigotimes}
\def\iy{\infty}
\def\ra{\rightarrow}
\def\longra{\longrightarrow}
\def\dashra{\dashrightarrow}
\def\t{\times}
\def\w{\wedge}
\def\bigw{\bigwedge}
\def\d{{\rm d}}
\def\bs{\boldsymbol}
\def\ci{\circ}
\def\ti{\tilde}
\def\ov{\overline}
\def\sv{\star\vp}
\title{On an Almost Contact Structure on $G_2$-Manifolds}

\author{A. J. Todd}

\address {Department of Mathematics \& Statistics, University of South Alabama, Mobile, AL, 36688}
\email{ajtodd@southalabama.edu}

\begin{abstract}
In this article, we study an almost contact metric structure on a $G_2$-manifold constructed by Arikan, Cho and Salur in \cite{ACS1} via the classification of almost contact metric structures given by Chinea and Gonzalez \cite{ChGo}. In particular, we characterize when this almost contact metric structure is cosymplectic and narrow down the possible classes in which this almost contact metric structure could lie. Finally, we show that any closed $G_2$-manifold admits an almost contact metric $3$-structure by constructing it explicitly and characterize when this almost contact metric $3$-structure is $3$-cosymplectic. 
\end{abstract}

\date{}
\maketitle
\section{Introduction}

Let $M$ be a smooth manifold of dimension $2n+1$. Let $\Xi$ denote a $2n$-dimensional distribution on $M$ which is maximally non-integrable. Locally, we have that $\Xi=\ker\al$ for some $1$-form $\al$; the ``maximally non-integrable'' condition then means that $a\w(\d\al)^n\neq 0$. If such a $\Xi$ exists, then the structure group of the tangent bundle to $M$ admits a reduction to the group $U(n)\t 1$. Conversely, if $M$ is an odd-dimensional manifold such that the structure group of the tangent bundle to $M$ admits a reduction to $U(n)\t 1$, then $M$ is said to have an \emph{almost contact structure}. Recently Borman, Eliashberg and Murphy \cite{BEM} have shown that in fact any closed, odd-dimensional manifold with an almost contact structure admits a contact structure. Our interest here is not in the contact structure itself, but rather the underlying almost contact structure. The reason is that such a reduction of the structure group of the tangent bundle to $M$ is equivalent to the existence of a $(1,1)$-
tensor $\phi$, a vector field $\xi$ and a $1$-form $\eta$ satisfying
\begin{equation}
 \phi^2=-I+\eta\ot\xi
\label{acs1}
\end{equation}
\begin{equation}
 \eta(\xi)=1
\label{acs2}
\end{equation}
Often, we will generally refer to the triple $(\phi,\xi,\eta)$ as the almost contact structure.  It is straight-forward to show that equations \eq{acs1} and \eq{acs2} imply that $\phi(\xi)\equiv 0$, $\eta\circ\phi\equiv 0$ and that $\phi$ has rank $2n$. A Riemannian metric $g$ is said to be \emph{compatible} with the almost contact structure $(\phi,\xi,\eta)$ if
\begin{equation}
 g(\phi X,\phi Y)=g(X,Y)-\eta(X)\eta(Y)
\label{acms1}
\end{equation}
for any vector fields $X$ and $Y$ on $M$. In this case, the quadruple $(\phi,\xi,\eta,g)$ is called an \emph{almost contact metric structure}. Using equations \eq{acs2} and \eq{acms1} and that $\phi\xi=0$ yields 
\begin{equation}
 g(\xi,X)=\eta(X)
\label{acms2}
\end{equation}
for any vector field $X$ on $M$. On any manifold with an almost contact structure, there always exists a compatible metric; however, such a compatible metric is not in general unique. Using $\phi$ and $g$, we can now define a smooth differential $2$-form $\om$ by
\begin{equation}
 \om(X,Y)=g(X,\phi Y)
\label{acms3}
\end{equation}
This $2$-form is called the \emph{fundamental $2$-form} of the almost contact metric structure. Note that $\om$ satisfies
\begin{equation}
 \eta\w\om^n\neq 0
\end{equation}
showing that almost contact metric manifolds are orientable. A reference for this material and more information on contact and almost contact geometry can be found in \cite{Blair2}. Various conditions can be placed on the almost contact metric structure $(\phi,\xi,\eta,g)$ to give various classes of almost contact metric structures such as cosymplectic, Sasakian and Kenmotsu geometries. These types of almost contact metric structures, and several others, have been studied extensively by many authors, see, e. g., \cite{Blair1}, \cite{Blair2}, \cite{BoGa}, \cite{CaNi}, \cite{CNY1}, \cite{CNY2}, \cite{GoYa}, \cite{IOV}, \cite{JaVa}, \cite{Kenm}, \cite{Kuo}, \cite{Li}, \cite{Ludd}, \cite{MaMu}, \cite{Olsz} and \cite{Puhl}. The classification of the various types of almost contact metric structures was undertaken by Chinea and Gonzalez \cite{ChGo} wherein they classified almost contact metric structures based on the covariant derivative of a certain differential $2$-form associated to the almost contact metric 
structure. We will discuss their classification in Section \ref{CACMS}.

Let $M$ be a $7$-dimensional manifold admitting a smooth differential $3$-form $\vp$ such that, for all $p\in M$, the pair $(T_pM,\vp)$ is isomorphic as an oriented vector space to the pair $(\R^7,\vp_0)$ where 
\begin{equation}
\vp_0=\d x^{123}+\d x^{145}+\d x^{167}+\d x^{246}-\d x^{257}-\d x^{347}-\d x^{356}
\label{g23form}
\end{equation}
with $\d x^{ijk}=\d x^i\w \d x^j\w \d x^k$. In \cite{Br1}, it is shown that the Lie group $G_2$ can be defined as the set of all elements of $GL(7,\R)$ that preserve $\vp_0$, so for a manifold admitting such a $3$-form, there is a reduction in the structure group of the tangent bundle to the exceptional Lie group $G_2$; hence, the pair $(M,\vp)$ is called a \emph{manifold with $G_2$-structure} whereby an abuse of terminology we often refer to $\vp$ as the $G_2$-structure. Using the theory of $G$-structures and the inclusion of $G_2$ in $SO(7)$, all manifolds with $G_2$-structure are necessarily orientable and spin, any orientable $7$-manifold with spin structure admits a $G_2$-structure, and associated to a given $G_2$-structure $\vp$ are a metric $g_{\vp}$ called the \emph{$G_2$-metric}, satisfying
\begin{equation}
(X\lrcorner\vp)\w(Y\lrcorner\vp)\w\vp=6g_{\vp}(X,Y)\Vol_{\vp}
\label{g2metric}
\end{equation}
for any vector fields $X$ and $Y$ on $M$ and a $2$-fold vector cross product $\t$; these three structures are related via
\begin{equation}
 \vp(X,Y,Z)=g_{\vp}(X\t Y, Z)
\label{g21}
\end{equation}
for all vector fields $X, Y, Z$. It is important to note that, one can also show that starting with either the two-fold vector cross product or the $G_2$-metric $g_{\vp}$, one can always obtain the other two structures (see e. g. \cite{Jo1} or \cite{FeGr}; thus in particular, if the flow of a vector field preserves one of these structures, it must also preserve the other two. A natural geometric requirement is that the $3$-form $\vp$ be covariant constant with respect to the Levi-Civita connection of the $G_2$-metric $g_{\vp}$; if this is so, we say that the $G_2$-structure is \emph{integrable} and call the pair $(M,\vp)$ a \emph{$G_2$-manifold}. Note that for a $G_2$-manifold $(M,\vp)$, the $2$-fold vector cross product $\t$ is also covariant-constant with respect to the Levi-Civita connection of the $G_2$-metric $g_{\vp}$. It is a nontrivial fact that the integrability of the $G_2$-structure is equivalent to the holonomy of $g_{\vp}$ being a subgroup of $G_2$ as well as $\vp$ being simultaneously closed 
and coclosed, that is, $\d\vp=0$ and $\de\vp=0$ respectively, where $\de$ is the adjoint to the exterior derivative $\d$ defined in terms of the Hodge star $\star$ of $g_{\vp}$. References for $G_2$-geometry and constructions of manifolds with $G_2$-holonomy include \cite{FeGr}, \cite{BrGr}, \cite{Br1}, \cite{Br2}, \cite{BrSa}, \cite{FeGr}, \cite{HaLa}, \cite{Jo1}, \cite{Jo2}, \cite{Kova}, \cite{Sa}, \cite{Verb} and \cite{Will}.

Note that, by Berger's classification of Riemannian holonomy groups, the only nontrivial subgroups of $G_2$ which can occur as the holonomy of a Riemannian manifold are $SU(2)$ and $SU(3)$. From \cite{Jo1} (or \cite{Jo2}), if $M$ is a $G_2$-manifold with holonomy equal to $SU(2)$, then $M$ is given by one of the Cartesian products $N\t\R^3$ or $N\t T^3$ where $N$ is a Calabi-Yau $2$-fold and $T^3$ is the $3$-torus; on the other hand, if $M$ is a $G_2$-manifold with holonomy equal to $SU(3)$, then $M$ is given by one of the Cartesian products $N\t\R$ or $N\t S^1$ where $N$ is a Calabi-Yau $3$-fold and $S^1$ is the unit circle. In this way, we see that there is a direct relation between $G_2$-geometry and symplectic geometry; however, in the case where $M$ has holonomy equal to $G_2$, we see that $G_2$-geometry really is its own geometry. An important property of $G_2$-manifolds with full $G_2$-holonomy is that they are simply connected, so in particular, the first Betti number of such a manifold is $0$. 
Moreover, all $G_2$-manifolds, regardless of holonomy, are Ricci-flat. It is well known that on a Ricci-flat manifold, the sets of parallel (with respect to the Levi-Civita connection) vector fields, Killing vector fields and harmonic vector fields coincide (a vector field is called \emph{harmonic} if its metric-dual $1$-form is harmonic in the sense of Hodge theory); hence, on a $G_2$-manifold with full $G_2$-holonomy there cannot exist any nonzero parallel vector fields and hence no nonzero Killing vector fields nor nonzero harmonic vector fields. Further, the existence of such a vector field (nonzero) immediately implies that the holonomy must either be trivial, $SU(2)$ or $SU(3)$.

With this juxtaposition, a natural question occurs. Do there exists $7$-dimensional manifolds which admit both a $G_2$-structure and a contact structure, and if so, how do these geometries interact? This particular line of research has only just begun. In the literature, there is work of Matzeu and Munteanu \cite{MaMu} on vector cross products and almost contact structures focusing in particular on hypersurfaces of $\R^8$, and a similar construction appears in \cite{Blair2}; however, work of Arikan, Cho and Salur \cite{ACS1}, \cite{ACS2} has really brought this issue to light. Using techniques of spin geometry, they show \cite[Cor. $3.2$]{ACS1} that every manifold with $G_2$-structure admits an almost contact structure. Note that this result does not depend on the manifold being closed, compact or noncompact nor does it depend on the integrability or non-integrability of the $G_2$-structure itself. Unfortunately, this result is nonconstructive; however, with this, we now have that there always exists at 
least one non-vanishing vector field, call it $\xi$, on a manifold with $G_2$-structure. Assume that $\xi$ has been normalized to have length $1$ everywhere using the $G_2$-metric $g_{\vp}$, then using $\xi$ and the two-fold vector cross product, define a $(1,1)$-tensor $\phi$ by
\begin{equation}
 \phi(X)=\xi\t X
\label{g2acms1}
\end{equation}
and a $1$-form $\eta$ by
\begin{equation}
 \eta(X)=g_{\vp}(\xi,X)
\end{equation}
Since the cross product on a $7$-manifold satisfies
\begin{equation}
 X_1\t(X_1\t X_2)=-g_{\vp}(X_1,X_1)X_2+g_{\vp}(X_1,X_2)X_1
\label{cp1}
\end{equation}
we immediately have that $\phi^2=-I+\eta\ot\xi$; that $\eta(\xi)=1$ follows immediately from the definitions. Thus, $(\phi,\xi,\eta)$ as defined here give an almost contact structure; moreover, this almost contact structure is compatible with $g_{\vp}$ since
\begin{equation}
 \begin{split}
  g_{\vp}(\phi X,\phi Y)&=g_{\vp}(\xi\t X,\xi\t Y)=\vp(\xi,X,\xi\t Y)=-\vp(\xi,\xi\t Y,X)=-g_{\vp}(\xi\t(\xi\t Y),X)\\
  &=-g_{\vp}(-Y+\eta(Y)\xi,X)=g_{\vp}(Y,X)-\eta(Y)g_{\vp}(\xi,X)=g_{\vp}(X,Y)-\eta(X)\eta(Y)\\
 \end{split}
\end{equation}
Hence, we see that $(\phi,\xi,\eta,g_{\vp})$ is an almost contact metric structure. Throughout we will refer to this as the \emph{standard $G_2$ almost contact metric structure}

Let $(M,\Xi)$ be a contact manifold where $\Xi$ is the $2n$-dimensional contact distribution. An almost contact metric structure $(\phi,\xi,\eta,g)$ is said to be the \emph{associated} almost contact metric structure of the contact structure if $\Xi=\ker\eta$ and $\om=\d\eta$ where $\om$ is given by \eq{acms3}. It is interesting to note, especially in light of the results from \cite{BEM}, that the above constructed almost contact metric structure on a closed $7$-manifold $M$ with $G_2$-structure cannot be the associated almost contact metric structure of a contact metric structure if the $G_2$-structure is closed, i. e., $\d\vp=0$ (see \cite[Corollary 5.8]{ACS1}). For if so
\begin{equation}
 \d\eta(X,Y)=g_{\vp}(X,\phi Y)=g(X,\xi\t_{\vp}Y)=\vp(\xi,Y,X)=-(\xi\lrcorner\vp)(X,Y)
\end{equation}
and hence, by equation \eq{g2metric} and an application of Stokes' Theorem,
\begin{equation}
\begin{split}
 Vol(M)&=\int_MVol_{\vp}=\int_M\frac{1}{6}((\xi\lrcorner\vp)\w(\xi\lrcorner\vp)\w\vp)\\
 &=\frac{1}{6}\int_M(\d\eta\w\d\eta\w\vp)=\frac{1}{6}\int_M\d(\eta\w\d\eta\w\vp)=\frac{1}{6}\int_{\partial M}(\eta\w\d\eta\w\vp)=0
\end{split}
\end{equation}
In particular, this is the case on a closed $G_2$-manifold.

Hence the goal of the current article is to further this line of research by shedding light on the classification of the standard $G_2$ almost contact structure focusing specifically on closed $G_2$-manifolds; also, this article is part of a larger research goal to use the well-established areas of symplectic and contact geometries to gain a better understanding of $G_2$-geometry, see \cite{CST1}, \cite{CST2}, \cite{Todd}. In Section $2$, we will discuss the classification scheme of Chinea and Gonzalez \cite{ChGo}; in Section $3$, we will use their classification to better understand the standard $G_2$ almost contact structure. In Section $4$, we give a brief introduction to almost contact metric $3$-structures from the work Kuo \cite{Kuo}. Finally, in Section $5$ we will construct an almost contact $3$-structure on a closed $7$-manifold with (possibly non-integrable) $G_2$-structure and use our results from Section $3$ to give some fundamental results on these structures. Our main results are as follows.

\begin{theo}
$\d\om=0$ if and only if $\na\xi=0$. Further, in this case, $\d\eta=0$, and the almost contact structure is normal and hence is cosymplectic.
\end{theo}

\begin{coro}
 If $(M,\vp)$ is a $G_2$-manifold with full $G_2$-holonomy, then $\d\om\neq 0$ (or equivalently, $\na\xi\neq0$). In particular, the standard $G_2$ almost contact metric structure cannot be cosymplectic, almost cosymplectic nor quasi-Sasakian.
\end{coro}

\begin{coro}
 If $(M,\vp)$ is a $G_2$-manifold, then the standard $G_2$ almost contact metric structure cannot be a contact metric structure. In particular, it cannot be Sasakian (or more generally, $a$-Sasakian where $a$ is constant).
\end{coro}

\begin{theo}
 Let $M$ be a $G_2$-manifold with the standard $G_2$ almost contact metric structure. If $\na\xi\neq 0$, then the almost contact metric structure is of class $\mathcal{G}$ where 
 \begin{enumerate}
  \item In general, we have
 \begin{equation*}
  \mathcal{G}\subseteq \mathcal{C}_5\op\mathcal{C}_6\op\mathcal{C}_7\op\mathcal{C}_8\op\mathcal{C}_9\op\mathcal{C}_{10}\op\mathcal{C}_{12}
 \end{equation*}
  \item In the case that $\de\eta=0$, then
 \begin{equation*}
  \mathcal{G}\subseteq \mathcal{C}_6\op\mathcal{C}_7\op\mathcal{C}_8\op\mathcal{C}_9\op\mathcal{C}_{10}\op\mathcal{C}_{12}
 \end{equation*}
  \item In the case that $\na_{\xi}\xi=0$, then
 \begin{equation*}
  \mathcal{G}\subseteq \mathcal{C}_5\op\mathcal{C}_6\op\mathcal{C}_7\op\mathcal{C}_8\op\mathcal{C}_9\op\mathcal{C}_{10}
 \end{equation*}
  \item In the case that both $\de\eta=0$ and $\na_{\xi}\xi=0$,
 \begin{equation*}
  \mathcal{G}\subseteq \mathcal{C}_6\op\mathcal{C}_7\op\mathcal{C}_8\op\mathcal{C}_9\op\mathcal{C}_{10}
 \end{equation*}
  \item In the case that the almost contact structure is normal, then
 \begin{equation*} 
  \mathcal{G} \subseteq \mathcal{C}_5\op\mathcal{C}_6\op\mathcal{C}_7\op\mathcal{C}_8
 \end{equation*}
  \item In the case that the almost contact structure is normal and $\de\eta=0$, then
 \begin{equation*}
  \mathcal{G} \subseteq \mathcal{C}_6\op\mathcal{C}_7\op\mathcal{C}_8
 \end{equation*}
 \end{enumerate}
\end{theo}

\begin{theo}
 Let $(M,\vp)$ be a closed $7$-manifold with (not necessarily integrable) $G_2$-structure $\vp$. Then $M$ admits an almost contact metric $3$-structure which is compatible with the $G_2$-metric.
\end{theo}

\begin{coro}
 On a closed $G_2$-manifold $(M,\vp)$ with full $G_2$-holonomy, the almost contact metric $3$-structure constructed above cannot be an almost cosymplectic $3$-structure and hence cannot be a cosymplectic $3$-structure.
\end{coro}

\begin{coro}
 On a closed $G_2$-manifold $(M,\vp)$ the almost contact metric $3$-structure constructed above cannot be a contact metric $3$-structure (equivalently, it cannot be $3$-Sasakian).
\end{coro}

\section{Classification of Almost Contact Metric Structures}
\label{CACMS}
In this section, we review the classification of almost contact metric structures given by Chinea and Gonzalez \cite{ChGo}.

As mentioned above, one can require many different types of conditions of a given almost contact metric structure. Of particular importance are normal almost contact structures. Let $(M,\phi,\xi,\eta)$ be an almost contact manifold. Then on $M\t\R$ we can define the almost complex structure $J$ given by 
\begin{equation}
 J\left(X,f\frac{\d}{\d t}\right)=\left(\phi X-f\xi, \eta(X)\frac{\d}{\d t}\right)
\end{equation}
Using the Nijenhuis torsion tensor defined for any $(1,1)$-tensor by
\begin{equation}
 [T,T](X,Y)=T^2[X,Y]-T[TX,Y]-T[X,TY]+[TX,TY]
\end{equation}
where the bracket on the right-hand side of the formula is the standard Lie bracket of vector fields, it is straight-forward to show that the integrability of $J$, that is, the vanishing of $[J,J]$, is equivalent to the simultaneous vanishing of the following four tensors on $M$:
\begin{equation}
 N^{(1)}(X,Y)=[\phi,\phi](X,Y)+2\d\eta(X,Y)\xi
\label{nacs1}
\end{equation}
where $[\phi,\phi]$ is the Nijenhuis torsion tensor of $\phi$;
\begin{equation}
 N^{(2)}(X,Y)=(\L_{\phi X}\eta)Y-(\L_{\phi Y}\eta)X
\label{nacs2}
\end{equation}
where $\L$ is the Lie derivative operator;
\begin{equation}
 N^{(3)}(X)=(\L_{\xi}\phi)X
\label{nacs3}
\end{equation}
and
\begin{equation}
 N^{(4)}=(\L_{\xi}\eta)X
\label{nacs4}
\end{equation}
Moreover, it can be shown that if $N^{(1)}$ is identically zero then $N^{(2)}$, $N^{(3)}$ and $N^{(4)}$ must be identically zero as well \cite{Blair2}. Thus, the integrability of $J$ is equivalent to the vanishing of the $(1,2)$-tensor $N^{(1)}$, and in this case, that is, if $N^{(1)}\equiv 0$, we call the almost contact structure $(\phi,\xi,\eta)$ \emph{normal}.

Other types of almost contact metric manifolds that have appeared in the literature are given as follows \cite{ChGo}:
\begin{itemize}
 \item {\it Almost cosymplectic:} $\d\om=0=\d\eta$
 \item {\it Cosymplectic:} almost cosymplectic and normal
 \item {\it Quasi-Sasakian:} $\d\om=0$ and normal
 \item {\it Almost $a$-Kenmotsu:} $\d\eta=0$ and $\d\om=\frac{2}{3}a\left[\eta(X)\om(Y,Z)+\eta(Y)\om(Z,X)+\eta(Z)\om(X,Y)\right]$ where $a$ is a differentiable function on $M$
 \item {\it $a$-Kenmotsu:} almost $a$-Kenmotsu and normal
 \item {\it Almost $a$-Sasakian:} $a\om=\d\eta$ where $a$ is a differentiable function on $M$
 \item {\it $a$-Sasakian:} Almost $a$-Sasakian and normal
 \item {\it Nearly $K$-cosymplectic:} $(\na_X\phi)Y+(\na_Y\phi)X=0$ and $\na_X\xi=0$ for all vector fields $X,Y$.
 \item {\it Quasi-$K$-cosymplectic:} $(\na_X\phi)Y+(\na_{\phi X}\phi)(\phi Y)=\eta(Y)\na_{\phi X}\xi$
 \item {\it Semi-cosymplectic:} $\de\om=0=\de\eta$
 \item {\it Trans-Sasakian:} \[(\na_X\om)(Y,Z)=-\frac{1}{2n}\left[\left(g(X,Y)\eta(Z)-g(X,Z)\eta(Y)\right)\de\om(\xi)+\left(g(X,\phi Y)\eta(Z)-g(X,\phi Z)\eta(Y)\right)\de\eta\right]\]
 \item {\it Nearly trans-Sasakian:} $(\na_X\om)(X,Y)=-\frac{1}{2n}\left[g(X,X)\de\om(Y)-g(X,Y)\de\om(X)+g(\phi X,Y)\eta(Y)\de\eta\right]$ and $(\na_X\eta)Y=-\frac{1}{2n}\left[g(\phi X,\phi Y)\de\eta+g(\phi X,Y)\de\om(\xi)\right]$
 \item {\it Almost $K$-contact:} $\na_{\xi}\phi=0$
\end{itemize}

The classification of Chinea and Gonzalez depends on a decomposition of the space $\mathcal{C}$ of covariant tensors of degree $3$ which possess the same symmetry properties as $\na\om$, the covariant derivative of the fundamental $2$-form \eq{acms3}. As we will need various pieces of their theory, we give here a sketch of the details of their decomposition and refer the reader to \cite{ChGo} for the full proof.

Let $V$ denote a real vector space of dimension $2n+1$ with an almost contact structure $(\phi,\xi,\eta)$ and a compatible metric $\langle\cdot,\cdot\rangle$. Then $\CV$ is a subspace of the space of covariant tensors of degree $3$, denoted $\bigot_3^0V$, defined by
\begin{equation}
 \mathcal{C}(V)=\left\{\al\in\bigot_3^0V:\al(x,y,z)=-\al(x,z,y)=-\al(x,\phi y,\phi z)+\eta(y)\al(x,\xi,z)+\eta(z)\al(x,y,\xi)\right\}
\end{equation}
Then the space of quadratic invariants of $\CV$ is generated by the following $18$ invariants:
\small
\begin{multicols}{2}
\begin{equation*}
\begin{split}
 &i_1(\al)=\sum_{i,j,k}\al(e_i,e_j,e_k)^2;\\
 &i_2(\al)=\sum_{i,j,k}\al(e_i,e_j,e_k)\al(e_j,e_i,e_k);\\
 &i_3(\al)=\sum_{i,j,k}\al(e_i,e_j,e_k)\al(\phi e_i,\phi e_j,e_k);\\
 &i_4(\al)=\sum_{i,j,k}\al(e_i,e_i,e_k)\al(e_j,e_j,e_k);\\
 &i_5(\al)=\sum_{j,k}\al(\xi,e_j,e_k)^2;\\
 &i_6(\al)=\sum_{i,k}\al(e_i,\xi,e_k)^2;\\
 &i_7(\al)=\sum_{j,k}\al(\xi,e_j,e_k)\al(e_j,\xi,e_k);\\
 &i_8(\al)=\sum_{i,j}\al(e_i,e_j,\xi)\al(e_j,e_i,\xi);\\
 &i_9(\al)=\sum_{i,j}\al(e_i,e_j,\xi)\al(\phi e_i,\phi e_j,\xi);\\
\end{split}
\end{equation*}

\begin{equation*}
\begin{split}
 &i_{10}(\al)=\sum_{i,j}\al(e_i,e_i,\xi)\al(e_j,e_j,\xi);\\
 &i_{11}(\al)=\sum_{i,j}\al(e_i,e_j,\xi)\al(e_j,\phi e_i,\xi);\\
 &i_{12}(\al)=\sum_{i,j}\al(e_i,e_j,\xi)\al(\phi e_j,\phi e_i,\xi);\\
 &i_{13}(\al)=\sum_{j,k}\al(\xi,e_j,e_k)\al(\phi e_j,\xi,e_k);\\
 &i_{14}(\al)=\sum_{i,j}\al(e_i,\phi e_i,\xi)\al(e_j,\phi e_j,\xi);\\
 &i_{15}(\al)=\sum_{i,j}\al(e_i,\phi e_i,\xi)\al(e_j,e_i,\xi);\\
 &i_{16}(\al)=\sum_{k}\al(\xi,\xi,e_k)^2;\\
 &i_{17}(\al)=\sum_{i,k}\al(e_i,e_i,e_k)\al(\xi,\xi,e_k);\\
 &i_{18}(\al)=\sum_{i,k}\al(e_i,e_i,\phi e_k)\al(\xi,\xi,e_k);\\
\end{split}
\end{equation*}
\end{multicols}
\normalsize
where $\{e_1,\ldots,e_{2n},\xi\}$ is an orthonormal basis of $V$ and $\al\in\CV$. They then give a decomposition of $\CV$ into orthogonal irreducible factors where orthogonality is with respect to the inner product given by
\begin{equation}
 \langle\al,\tilde{\al}\rangle=\sum_{i,j,k=1}^{2n+1}\al(e_i,e_j,e_k)\tilde{\al}(e_i,e_j,e_k)
\end{equation}
where $\al,\tilde{\al}\in\CV$ and $\{e_i\}$ is an orthonormal basis of $V$. First, they decompose $\CV$ into three orthogonal subspaces $\mathcal{D}_1$, $\mathcal{D}_2$ and $\mathcal{C}_{12}$ defined by
\begin{equation}
 \mathcal{D}_1=\{\al\in\CV:\al(\xi,x,y)=\al(x,\xi,y)=0\}
\label{d1}
\end{equation}
\begin{equation}
 \mathcal{D}_2=\{\al\in\CV:\al(x,y,z)=\eta(x)\al(\xi,y,z)+\eta(y)\al(x,\xi,z)+\eta(z)\al(x,y,\xi)\}
\label{d2}
\end{equation}
\begin{equation}
 \mathcal{C}_{12}=\{\al\in\CV:\al(x,y,z)=\eta(x)\eta(y)\al(\xi,\xi,z)+\eta(x)\eta(z)\al(\xi,y,\xi)\}
\label{c12}
\end{equation}
Next, they further decompose $\mathcal{D}_1$ into four orthogonal irreducible subspaces given by
\begin{equation}
 \mathcal{C}_1=\{\al\in\CV:\al(x,x,y)=\al(x,y,\xi)=0\}
\label{c1}
\end{equation}
\begin{equation}
 \mathcal{C}_2=\{\al\in\CV:\al(x,y,z)+\al(y,z,x)+\al(z,x,y)=0\text{, }\al(x,y,\xi)=0\}
\label{c2}
\end{equation}
\begin{equation}
 \mathcal{C}_3=\{\al\in\CV:\al(x,y,z)-\al(\phi x,\phi y,z)=0\text{, }c_{12}\al=0\}
\label{c3}
\end{equation}
\begin{equation}
 \begin{split}
 \mathcal{C}_4=\{\al\in\CV:&\al(x,y,z)=\frac{1}{2(n-1)}[(\langle x,y\rangle-\eta(x)\eta(y))c_{12}\al(z)-(\langle x,z\rangle-\eta(x)\eta(z))c_{12}\al(y)\\
 &-\langle x,\phi y\rangle c_{12}\al(\phi z)+\langle x,\phi z\rangle c_{12}\al(\phi y)]\text{, }c_{12}\al(\xi)=0\}\\
 \end{split}
\end{equation}
for any $x,y,z\in V$ and $c_{12}\al(x)=\sum\al(e_i,e_i,x)$ where $\{e_i\}$ is an orthonormal basis of $V$. Finally, they decompose $\mathcal{D}_2$ into seven orthogonal irreducible subspaces given by
\begin{equation}
 \mathcal{C}_5=\{\al\in\CV:\al(x,y,z)=\frac{1}{2n}[\langle x,\phi z\rangle\eta(y)\ov{c}_{12}\al(\xi)-\langle x,\phi y\rangle\eta(z)\ov{c}_{12}\al(\xi)]\}
\label{c5}
\end{equation}
\begin{equation}
 \mathcal{C}_6=\{\al\in\CV:\al(x,y,z)=\frac{1}{2n}[\langle x,y\rangle\eta(z)c_{12}\al(\xi)-\langle x,z\rangle\eta(y)c_{12}\al(\xi)]\}
\label{c6}
\end{equation}
\begin{equation}
 \mathcal{C}_7=\{\al\in\CV:\al(x,y,z)=\eta(z)\al(y,x,\xi)-\eta(y)\al(\phi x,\phi z,\xi)\text{, }c_{12}\al(\xi)=0\}
\label{c7}
\end{equation}
\begin{equation}
 \mathcal{C}_8=\{\al\in\CV:al(x,y,z)=-\eta(z)\al(y,x,\xi)-\eta(y)\al(\phi x,\phi z,\xi)\text{, }\ov{c}_{12}\al(\xi)=0\}
\label{c8}
\end{equation}
\begin{equation}
 \mathcal{C}_9=\{\al\in\CV:\al(x,y,z)=\eta(z)\al(y,x,\xi)+\eta(y)\al(\phi x,\phi z,\xi)\}
\label{c9}
\end{equation}
\begin{equation}
 \mathcal{C}_{10}=\{\al\in\CV:\al(x,y,z)=-\eta(z)\al(y,x,\xi)+\eta(y)\al(\phi x,\phi z,\xi)\}
\label{c10}
\end{equation}
\begin{equation}
 \mathcal{C}_{11}=\{\al\in\CV:\al(x,y,z)=-\eta(x)\al(\xi,\phi y,\phi z)\}
\label{c11}
\end{equation}
for any $x,y,z\in V$ where $\ov{c}_{12}(\xi)=\sum\al(e_i,\phi e_i,\xi)$ and $\{e_i\}$ is an orthonormal basis of $V$. Thus, we have $12$ invariant subspaces $\mathcal{C}_i$ of $\CV$ which are orthogonal and irreducible with respect to the $U(n)\t 1$ action.

Using the definition of the inner product on $\CV$ together the quadratic invariants, we see that
\begin{equation}
 ||\al||^2=i_1(\al)+i_5(\al)+2i_6(\al)+i_{16}(\al)
\end{equation}
\begin{equation}
 ||c_{12}(\al)||^2=i_4(\al)+i_{10}(\al)+i_{16}(\al)+2i_{17}(\al)
\end{equation}
\begin{equation}
 ||\ov{c}_{12}(\al)||^2=i_4(\al)+i_{14}(\al)
\label{codiffeta}
\end{equation}
Then there exist linear relations among the quadratic invariants for each of the subspaces $C_{i}$ which are used to prove that each of the subspaces is irreducible. These relations will prove useful to us, so we give them below. Let $A=\{1,2,3,4,5,7,11,13,15,16,17,18\}$.
 \begin{equation*}
  \begin{split}
   \mathcal{C}_1\text{: }\text{ }\text{ }&i_1(\al)=-i_2(\al)=-i_3(\al)=||\al||^2\text{; }i_m(\al)=0\text{ for }m\geq 4;\\
   \mathcal{C}_2\text{: }\text{ }\text{ }&i_1(\al)=2i_2(\al)=-i_3(\al)=||\al||^2\text{; }i_m(\al)=0\text{ for }m\geq 4;\\
   \mathcal{C}_3\text{: }\text{ }\text{ }&i_1(\al)=i_3(\al)=||\al||^2\text{; }i_2(\al)=i_m(\al)=0\text{ for }m\geq 4;\\
   \mathcal{C}_4\text{: }\text{ }\text{ }&i_1(\al)=i_3(\al)=\frac{n}{(n-1)^2}i_4(\al)=\frac{n}{(n-1)^2}\sum_k^{2n}c_{12}^2(\al)(e_k)\text{; }i_2(\al)=i_m(\al)=0\text{ for }m>4;\\
   \mathcal{C}_5\text{: }\text{ }\text{ }&i_6(\al)=-i_8(\al)=i_9(\al)=-i_{12}(\al)=\frac{1}{2n}i_{14}(\al)\text{; }i_{10}(\al)=i_m(\al)=0\text{ for }m\in A;\\
   \mathcal{C}_6\text{: }\text{ }\text{ }&i_6(\al)=i_8(\al)=i_9(\al)=i_{12}(\al)=\frac{1}{2n}i_{10}(\al)\text{; }i_{14}(\al)=i_m(\al)=0\text{ for }m\in A;\\
   \mathcal{C}_7\text{: }\text{ }\text{ }&i_6(\al)=i_8(\al)=i_9(\al)=-i_{12}(\al)=\frac{||\al||^2}{2}\text{; }i_{10}(\al)=i_{14}(\al)=i_m(\al)\text{ for }m\in A;\\
   \mathcal{C}_8\text{: }\text{ }\text{ }&i_6(\al)=-i_8(\al)=i_9(\al)=-i_{12}(\al)=\frac{||\al||^2}{2}\text{; }i_{10}(\al)=i_{14}(\al)=i_m(\al)\text{ for }m\in A;\\
   \mathcal{C}_9\text{: }\text{ }\text{ }&i_6(\al)=i_8(\al)=-i_9(\al)=-i_{12}(\al)=\frac{||\al||^2}{2}\text{; }i_{10}(\al)=i_{14}(\al)=i_m(\al)\text{ for }m\in A;\\
   \mathcal{C}_{10}\text{: }\text{ }\text{ }&i_6(\al)=-i_8(\al)=-i_9(\al)=i_{12}(\al)=\frac{||\al||^2}{2}\text{; }i_{10}(\al)=i_{14}(\al)=i_m(\al)\text{ for }m\in A;\\
   \mathcal{C}_{11}\text{: }\text{ }\text{ }&i_5(\al)=||\al||^2\text{; }i_m(\al)=0\text{ for }m\neq 5;\\
   \mathcal{C}_{12}\text{: }\text{ }\text{ }&i_{16}(\al)=||\al||^2\text{; }i_m(\al)=0\text{ for }m\neq 16;\\
  \end{split}
 \end{equation*}

Now, let $M$ be an odd-dimensional manifold with an almost contact metric structure; then, for all $p\in M$, $T_pM$ is an odd-dimensional vector space with the almost contact metric structure $(\phi_p,\xi_p,\eta_p,g_p)$. Hence, we can decompose the vector space $\mathcal{C}(T_pM)$ according to the decomposition above, and we say that the almost contact structure is of class $\mathcal{C}_j$, where $\mathcal{C}_j$ is one of the invariant subspaces of $\mathcal{C}(T_pM)$, if $(\na\om)_p\in\mathcal{C}_j$ for all $p\in M$. Note that $\al(x,y,z)$ from above, becomes $(\na_X\om)(Y,Z)$ for vector fields $X,Y,Z$ on $M$, that $c_{12}(\na\om)(X)=-\de\om(X)$ and that $\ov{c}_{12}(\na\om)(\xi)=\de\eta$ where we have used the formula \begin{equation}
\de\cdot=-\sum_{i=1}^ne_i\lrcorner(\na_{e_i}\cdot)
\end{equation}
for the codifferential where $\{e_i\}$ is an orthonormal frame.

Chinea and Gonzalez then give the relationship between the classes $\mathcal{C}_i$ and the various types of almost contact metric manifolds defined above. Again, we will refer to various of these relationships in our classification, so we will give them here.
\begin{itemize}
 \item {\it Almost cosymplectic:} $\mathcal{C}_2\op\mathcal{C}_9$
 \item {\it Quasi-Sasakian:} $\mathcal{C}_6\op\mathcal{C}_7$
 \item {\it $a$-Kenmotsu:} $\mathcal{C}_5$
 \item {\it $a$-Sasakian:} $\mathcal{C}_6$
 \item {\it Nearly $K$-cosymplectic:} $\mathcal{C}_1$
 \item {\it Quasi-$K$-cosymplectic:} $\mathcal{C}_1\op\mathcal{C}_2\op\mathcal{C}_9\op\mathcal{C}_{10}$
 \item {\it Semi-cosymplectic:} $\mathcal{C}_1\op\mathcal{C}_2\op\mathcal{C}_3\op\mathcal{C}_7\op\mathcal{C}_8\op\mathcal{C}_9\op\mathcal{C}_{10}\op\mathcal{C}_{11}\op\mathcal{C}_{12}$
 \item {\it Trans-Sasakian:} $\mathcal{C}_5\op\mathcal{C}_6$
 \item {\it Nearly trans-Sasakian:} $\mathcal{C}_1\op\mathcal{C}_5\op\mathcal{C}_6$
 \item {\it Almost $K$-contact:} $\bigop_{j=1}^{10}\mathcal{C}_j$
 \item {\it Normal:} $\mathcal{C}_3\op\mathcal{C}_4\op\mathcal{C}_5\op\mathcal{C}_6\op\mathcal{C}_7\op\mathcal{C}_8$
 \end{itemize}


\section{Classification Theorems for the Standard $G_2$ Almost Contact Metric Structure on $G_2$-Manifolds}

Throughout this section $(M,\vp)$ will denote a closed, i. e., compact with empty boundary, $G_2$-manifold with the standard $G_2$ almost contact structure $(\phi,\xi,\eta)$ and $\{e_1,e_2,e_3,e_4,e_5,e_6,\xi\}$ will denote a local orthonormal frame about an arbitrary point. Also, we denote the $G_2$-metric $g_{\vp}$ simply by $g$ as it is the only metric under consideration in this section.

We begin by noting some helpful formulae. Recall that on a $G_2$-manifold, the cross product is parallel, so 
\begin{equation}
 (\na_X\phi)(Y)=\na_X(\phi Y)-\phi(\na_XY)=\na_X(\xi\t Y)-\xi\t(\na_XY)=\na_X\xi\t Y
\label{nablaphi}
\end{equation}
In particular, $\na\phi=0$ if and only if $\na\xi=0$. Further, since $g(\xi,\xi)=1$, we find
\begin{equation}
 0=\na_Xg(\xi,\xi)=2g(\na_X\xi,\xi)
\end{equation}
which implies that $\na_X\xi$ is perpendicular to $\xi$ for all vector fields $X$ on $M$, and finally, for any vector fields $X,Y,Z$ on $M$, we have
\begin{equation}
\begin{split}
 (\na_X\om)(Y,Z)&=\na_X\om(Y,Z)-\om(\na_XY,Z)-\om(Y,\na_XZ)=\na_Xg(Y,\phi Z)-g(\na_XY,\phi Z)-g(Y,\phi(\na_XZ))\\
 &=g(\na_XY,\phi Z)+g(Y,(\na_X\phi)Z)+g(Y,\phi(\na_XZ))-g(\na_XY,\phi Z)-g(Y,\phi(\na_XZ))\\
 &=g(Y,\na_X\xi\t Z)\\
\end{split}
\end{equation}

\begin{thm}
$\d\om=0$ if and only if $\na\xi=0$. Further, in this case, $\d\eta=0$, and the almost contact structure is normal and hence is cosymplectic.
\label{cosymplectic}
\end{thm}

\begin{proof}
 We note that
 \begin{equation}
  \om(X,Y)=g(X,\phi Y)=g(X, \xi\t Y)=g(\xi\t Y, X)=\vp(\xi,Y,X)=-\vp(\xi,X,Y)=(-\xi\lrcorner\vp)(X,Y)
 \end{equation}
 Thus, using the fact that $\d\vp=0$, we see that
 \begin{equation}
  -d\om=\d(\xi\lrcorner\vp)=\xi\lrcorner\d\vp+\d(\xi\lrcorner\vp)=\L_{\xi}\vp
 \end{equation}
 From this, we conclude that $\d\om=0$ if and only if $\L_{\xi}\vp=0$; recall from the Introduction we have $\L_{\xi}\vp=0$ if and only if $\L_{\xi}g=0$, that is, if and only if $\xi$ is Killing with respect to the $G_2$-metric. Since $G_2$-manifolds are Ricci-flat, $\xi$ is Killing if and only if $\na\xi=0$.
 
 We assume for the remainder of this proof that $\na\xi=0$. Then 
 \begin{equation}
 \begin{split}
  2\d\eta(X,Y)&=X\eta(Y)-Y\eta(X)-\eta([X,Y])=Xg(\xi,Y)-Yg(\xi,X)-g(\xi,[X,Y])\\
  &=g(\na_X\xi,Y)+g(\xi,\na_XY)-g(\na_Y\xi,X)-g(\xi,\na_YX)-g(\xi,\na_XY-\na_YX)\\
  &=g(\xi,\na_XY-\na_YX-\na_XY+\na_YX)=0\\
 \end{split}
 \end{equation}
 
 Finally, to see that in this case the almost contact metric structure is normal, we first recall that $\na\xi=0$ implies $\na\phi=0$. Then calculating
\begin{equation}
 \begin{split}
  [\phi,\phi](X,Y)&+2\d\eta(X,Y)\xi\\
  =&\phi^2[X,Y]-\phi[\phi X,Y]-\phi[X,\phi Y]+[\phi X,\phi Y]\\
  =&-[X,Y]+\eta([X,Y])\xi-\phi[\phi X,Y]-\phi[X,\phi Y]+[\phi X,\phi Y]\\
  =&-[X,Y]+g(\xi,[X,Y])\xi-\phi[\phi X,Y]-\phi[X,\phi Y]+[\phi X,\phi Y]\\
  =&-[X,Y]+g(\xi,[X,Y])\xi-\phi\{\na_{\phi X}Y-(\na_Y\phi)X-\phi\na_YX\}-\phi\{(\na_X\phi)Y+\phi\na_XY-\na_{\phi Y}X\}\\
  &+(\na_{\phi X}\phi)Y+\phi\na_{\phi X}Y-(\na_{\phi Y}\phi)X-\phi\na_{\phi Y}X\\
  =&-[X,Y]+g(\xi,[X,Y])\xi-\phi\na_{\phi X}Y+\phi^2\na_YX-\phi^2\na_XY+\phi\na_{\phi Y}X+\phi\na_{\phi X}Y-\phi\na_{\phi Y}X\\
  =&-[X,Y]+g(\xi,[X,Y])\xi-\na_YX+\eta(\na_YX)\xi+\na_XY-\eta(\na_XY)\xi\\
  =&-\na_XY+\na_YX+\na_XY-\na_YX+g(\xi,\na_XY-\na_YX+\na_YX-\na_XY)\xi=0\\
  \end{split}
 \end{equation} 
\end{proof}

\begin{cor}
 If $(M,\vp)$ is a $G_2$-manifold with full $G_2$-holonomy, then $\d\om\neq 0$ (or equivalently, $\na\xi\neq0$). In particular, the standard $G_2$ almost contact metric structure cannot be cosymplectic, almost cosymplectic nor quasi-Sasakian.
\label{G2manifold}
\end{cor}

\begin{proof}
 The first statement follows since the existence of a parallel vector field on a $G_2$-manifold implies a reduction in the holonomy group to a proper subgroup of $G_2$. The second statement follows since all three geometries require $\d\om=0$.
\end{proof}

\begin{cor}
 If $(M,\vp)$ is a $G_2$-manifold, then the standard $G_2$ almost contact metric structure cannot be a contact metric structure. In particular, it cannot be Sasakian (or more generally, $a$-Sasakian where $a$ is constant).
\label{sasakian}
\end{cor}

\begin{proof}
 Recall that a contact metric structure requires $\om=\d\eta$ which immediately yields $\d\om=0$; however, by the above theorem, this implies that $\d\eta=0$ which gives $\om=0$, contradicting the fact that $\eta\w(\om)^3\neq 0$.
\end{proof}

Note that Chinea and Gonzalez' result means that we can decompose $\na\om$ as follows:
\begin{equation}
 \begin{split}
  (\na_X\om)(Y,Z)&=\al_1(X,Y,Z)+\al_2(X,Y,Z)+\be_{12}(X,Y,Z)\\
  &=(\be_1(X,Y,Z)+\be_2(X,Y,Z)+\be_3(X,Y,Z)+\be_4(X,Y,Z))+(\be_5(X,Y,Z)+\be_6(X,Y,Z)\\
  &+\be_7(X,Y,Z)+\be_8(X,Y,Z)+\be_9(X,Y,Z)+\be_{10}(X,Y,Z)+\be_{11}(X,Y,Z))+\be_{12}(X,Y,Z)\\
 \end{split}
\end{equation}
where $\al_1\in\mathcal{D}_1$, $\al_2\in\mathcal{D}_2$, $\be_i\in\mathcal{C}_i$ for all $i=1,2,\ldots,12$ with $\al_1=\be_1+\be_2+\be_3+\be_4$ and $\al_2=\sum_{i=5}^{11}\be_{i}$. Since we have classified the almost contact metric structures where $\xi$ is parallel, we now consider the case that $\na\xi\neq 0$. We begin with the following observation.

\begin{lem}
 $(\na_X\om)(\xi,Y)=0$ if and only if $\na\xi=0$.
\end{lem}

\begin{proof}
 We note that 
 \begin{equation}
  (\na_X\om)(\xi,Y)=g(\xi,\na_X\xi\t Y)=g(\xi\t\na_X\xi,Y)
 \end{equation}
Thus, $(\na_X\om)(\xi,Y)=0$ for all vector fields $X,Y$ if and only if $\xi\t\na_X\xi=0$ for all vector fields $X$. Of course, if $\na_X\xi=0$ for all vector fields $X$, then $\xi\t\na_X\xi=0$ for all vector fields $X$. Conversely, if $\xi\t\na_X\xi=0$ for all vector fields $X$, then
\begin{equation}
 0=\xi\t(\xi\t\na_X\xi)=-\na_X\xi+g(\xi,\na_X\xi)\xi=-\na_X\xi
\end{equation}
where the first equality follows from \eq{cp1} and the second follows from the fact that $\na_X\xi$ is perpendicular to $\xi$ for all $X$, and hence the result follows.
\end{proof}

Since $\al_1$ requires in particular that $\al_1(X,\xi,Y)=0$ for all $X,Y$, this immediately yields the following.

\begin{cor}
 If $\na\xi\neq 0$, then $\al_1=0$.
\end{cor}

Now, that $\na\xi\neq 0$ implies that $\xi$ is not a harmonic vector field, so in particular, we must have either $\d\eta\neq0$ or $\de\eta\neq 0$ or both. Note that since $||\de\eta||^2$ at each point is given by $i_4(\na\om)+i_{14}(\na\om)$, $\de\eta=0$ for all $p\in M$, for all classes $\mathcal{C}_j$ except when $j=4,5$; since we already have that $\be_4=0$ because $\al_1=0$, we see that, if $\na\xi\neq 0$, then $\de\eta\neq0$ if and only if $\be_5\neq 0$. 

To continue with the classification, we note that under the condition that $\na\xi\neq 0$, there are two natural cases: either $\na_{\xi}\xi=0$ (in which case there exists some vector field $X\neq\xi$ with $\na_{X}\xi\neq 0$) or $\na_{\xi}\xi\neq 0$. The following lemmas will prove useful.

\begin{lem}
 $i_6(\na\om)=0$ if and only if $\na_{e_j}\xi=0$ for all $j=1,\ldots,6$.
\end{lem}

\begin{proof}
 Note that, from the definitions of the quadratic invariants given in Section \ref{CACMS}
 \begin{equation}
  0=i_6(\na\om)=\sum_{j,k=1}^6\left((\na_{e_j}\om)(\xi,e_k)\right)^2
 \end{equation}
if and only if for all $j,k=1,\ldots,6$
\begin{equation}
 0=(\na_{e_j}\om)(\xi,e_k)=g(\xi,\na_{e_j}\xi\t e_k)=g(\na_{e_j}\xi\t e_k,\xi)=\vp(\na_{e_j}\xi,e_k,\xi)=\vp(\xi,\na_{e_j}\xi,e_k)=g(\xi\t\na_{e_j}\xi,e_k)
\end{equation}
Recalling that $\xi\t\na_{e_j}\xi$ is also perpendicular to $\xi$, we see, by nondegeneracy of the metric, that $g(\xi\t\na_{e_j}\xi,e_k)=0$ for all $j,k=1,\ldots,6$ if and only if for all $j=1,\ldots,6$
\begin{equation}
 \xi\t\na_{e_j}\xi=0
\end{equation}
Since $\xi\t(\xi\t\na_{e_j}\xi)=-\na_{e_j}\xi+g(\xi,\na_{e_j}\xi)\xi=-\na_{e_j}\xi$, we thus find that $\xi\t\na_{e_j}\xi=0$ for all $j=1,\ldots,6$ if and only if $\na_{e_j}\xi=0$ for all $j=1,\ldots, 6$ as we wanted to show.
\end{proof}

Using similar techniques, we can show

\begin{lem}
 $i_{16}(\na\om)=\sum_{k}\left((\na_{\xi}\om)(\xi,e_k)\right)^2=0$ if and only if $\na_{\xi}\xi=0$.
\end{lem}

From this, we see that $\na_{\xi}\xi\neq 0$ if and only if $i_{16}\neq 0$; however, the only class that allows $i_{16}\neq 0$ is $\mathcal{C}_{12}$ and hence we conclude that $\be_{12}\neq 0$ if and only if $\na_{\xi}\xi\neq 0$. Further, we note that since class $\mathcal{C}_{11}$ requires both $i_5(\be_{11})\neq 0$ and $i_{16}(\be_{11})=0$; however, by the above, $i_{16}(\be_{11})=0$ implies $\na_{\xi}\xi=0$, so 
\begin{equation}
 i_5(\be_{11})=\sum_{j,k}g(e_j,\na_{\xi}\xi\t e_k)=0
\end{equation}
Hence, we must have $\be_{11}=0$ always. Summarizing these results, we have the following theorem.

\begin{thm}
 Let $M$ be a $G_2$-manifold with the standard $G_2$ almost contact metric structure. If $\na\xi\neq 0$, then the almost contact metric structure is of class $\mathcal{G}$ where 
 \begin{enumerate}
  \item In general, we have
 \begin{equation}
  \mathcal{G}\subseteq \mathcal{C}_5\op\mathcal{C}_6\op\mathcal{C}_7\op\mathcal{C}_8\op\mathcal{C}_9\op\mathcal{C}_{10}\op\mathcal{C}_{12}
 \end{equation}
  \item In the case that $\de\eta=0$, then
 \begin{equation}
  \mathcal{G}\subseteq \mathcal{C}_6\op\mathcal{C}_7\op\mathcal{C}_8\op\mathcal{C}_9\op\mathcal{C}_{10}\op\mathcal{C}_{12}
 \end{equation}
  \item In the case that $\na_{\xi}\xi=0$, then
 \begin{equation}
  \mathcal{G}\subseteq \mathcal{C}_5\op\mathcal{C}_6\op\mathcal{C}_7\op\mathcal{C}_8\op\mathcal{C}_9\op\mathcal{C}_{10}
 \end{equation}
  \item In the case that both $\de\eta=0$ and $\na_{\xi}\xi=0$,
 \begin{equation}
  \mathcal{G}\subseteq \mathcal{C}_6\op\mathcal{C}_7\op\mathcal{C}_8\op\mathcal{C}_9\op\mathcal{C}_{10}
 \end{equation}
  \item In the case that the almost contact structure is normal, then
 \begin{equation} 
  \mathcal{G} \subseteq \mathcal{C}_5\op\mathcal{C}_6\op\mathcal{C}_7\op\mathcal{C}_8
 \end{equation}
  \item In the case that the almost contact structure is normal and $\de\eta=0$, then
 \begin{equation}
  \mathcal{G} \subseteq \mathcal{C}_6\op\mathcal{C}_7\op\mathcal{C}_8
 \end{equation}
 \end{enumerate}
\end{thm}

\section{Almost Contact $3$-Structures}
\label{AC3S}

We now define almost contact $3$-structures as defined originally by Kuo \cite{Kuo}. Let $M$ be a smooth manifold of dimension $4m+3$ \cite[Thm. $4$]{Kuo}. Then an \emph{almost contact $3$-structure} on $M$ is a reduction of the structure group of the tangent bundle of $M$ to $Sp(m)\t I_3$ where $I_3$ is the $3\t 3$ identity matrix \cite[Thm. $5$]{Kuo}. Such a reduction of the structure group of the tangent bundle of $M$ is equivalent to the existence of $3$ almost contact structures $(\phi_i,\xi_i,\eta_i)$, $i=1,2,3$ satisfying
\begin{equation}
 \eta_i(\xi_j)=\eta_j(\xi_i)=0
\label{ac3s1}
\end{equation}
\begin{equation}
 \phi_i\xi_j=-\phi_j\xi_i=\xi_k
\label{ac3s2}
\end{equation}
\begin{equation}
 \eta_i\circ\phi_j=-\eta_j\circ\phi_i=\eta_k
\label{ac3s3}
\end{equation}
\begin{equation}
 \phi_i\phi_j-\eta_j\ot\xi_i=-\phi_j\phi_i+\eta_i\ot\xi_j=\phi_k
\label{ac3s4}
\end{equation}
for any cyclic permutation $(i,j,k)$ of $(1,2,3)$. Of fundamental importance is that if $M^{4m+3}$ admits $2$ almost contact structures $(\phi_i,\xi_i,\eta_i)$, $i=1,2$ satisfying
\begin{equation}
 \eta_1(\xi_2)=\eta_2(\xi_1)=0
\label{ac3s5}
\end{equation}
\begin{equation}
 \phi_1\xi_2=-\phi_2\xi_1
\label{ac3s6}
\end{equation}
\begin{equation}
 \eta_1\circ\phi_2=-\eta_2\circ\phi_1
\label{ac3s7}
\end{equation}
\begin{equation}
 \phi_1\phi_2-\eta_2\ot\xi_1=-\phi_2\phi_1+\eta_1\ot\xi_2
\label{ac3s8}
\end{equation}
then in fact $M$ admits an almost contact $3$-structure by setting $\phi_3=\phi_1\phi_2-\eta_2\ot\xi_1$, $\xi_3=\phi_1\xi_2$ and $\eta_3=\eta_1\circ\phi_2$; see \cite[Thm. $1$]{Kuo}. If there exists a single Riemannian metric $g$ that is compatible with all $3$ of the almost contact structures, then $(\phi_i,\xi_i,\eta_i,g)$, $i=1,2,3$ is called an \emph{almost contact metric $3$-structure}. Again, of fundamental importance is that if $M^{4m+3}$ admits $2$ almost contact metric structures $(\phi_i,\xi_i,\eta_i,g)$, $i=1,2$, then \eq{ac3s5}, \eq{ac3s6} and \eq{ac3s7} follow from \eq{ac3s8}; see \cite[Thm. $2$]{Kuo}. Almost contact $3$-structures were introduced in order to give a structure of contact-type that is similar to an almost quaternionic structure in the same way that an almost contact structure is similar to an almost complex structure. Almost contact $3$-structures have been the subject of many articles and have been shown to have importance in physics as well.

As with almost contact metric structures, various types of almost contact metric $3$-structures have been of particular interest to several authors, see e. g., \cite{CaNi}, \cite{CNY2}, \cite{IOV}. In particular, a \emph{$3$-cosymplectic structure} is an almost contact metric $3$-structure $(\phi_i,\xi_i,\eta_i,g)$, $i=1,2,3$, where each of the almost contact metric structures is cosymplectic.

\section{An Almost Contact Metric $3$-Structure on Closed Manifolds with a $G_2$-Structure}

Throughout this section, let $M$ denote a closed, orientable $7$-manifold equipped with a (not necessarily integrable) $G_2$-structure $\vp$. Let $g=g_{\vp}$ denote the $G_2$-metric of the $G_2$-structure $\vp$. Results of Thomas show that on $M$, there exist two, linearly independent, non-vanishing vector fields, call them $u$ and $v$. Assume that $u$ and $v$ have been normalized to have length $1$ using the $G_2$-metric $g$. Using the cross product of the $G_2$-structure, we in fact get a third non-vanishing vector field $u\t v$; that $\{u,v,u\t v\}$ is linearly independent at each point follows from the fact that $u\t v$ is orthogonal to both $u$ and $v$ at each point. Since there is no reason to assume that $u$ and $v$ are themselves orthogonal at each point, we will use $u$ and $u\t v$ to construct the almost contact metric $3$-structure. To do this, we first define $(\phi_i,\xi_i,\eta_i)$, $i=1,2$ as follows:
\begin{equation}
 \xi_1=u\text{, }\text{ }\text{ }\text{ }\xi_2=\frac{u\t v}{||u\t v||}
\end{equation}
where $||\cdot||$ is the norm of the $G_2$-metric;
\begin{equation}
 \phi_i(X)=\xi_i\t X
\end{equation}
and
\begin{equation}
 \eta_i(X)=g(\xi_i,X)
\end{equation}
for any vector field $X$ on $M$ and $i=1,2$. Then as in Arikan, Cho and Salur's construction, $(\phi_i,\xi_i,\eta_i)$, $i=1,2$ defines two almost contact structures, both of which are compatible with the $G_2$-metric $g$. By the results of Kuo mentioned in Section \ref{AC3S}, it is sufficient to show that $\phi_1\phi_2-\eta_2\ot\xi_1=-\phi_2\phi_1+\eta_1\ot\xi_2$ is satisfied to prove that $M$ admits an almost contact metric $3$-structure. Let $f=1/||u\t v||$. Then, since the cross product satisfies the relation $u\t((u\t v)\t X)=-(u\t v)\t(u\t X)+u\t(u\t(v\t X))+(u\t(u\t X))\t v$ for any vector field $X$, we have
 \begin{equation}
  \begin{split}
   \phi_1&\phi_2(X)-\eta_2(X)\xi_1=f\phi_1((u\t v)\t X)-fg(u\t v,X)u\\
   =&f(u\t((u\t v)\t X))-fg(u\t v,X)u\\
   =&f\{-(u\t v)\t(u\t X)+(u\t(u\t(v\t X)))+(u\t(u\t X))\t v\}-fg(u\t v,X)u\\
   =&f\{-(u\t v)\t(u\t X)-v\t X+g(u,v\t X)u+(-X+g(u,X)u)\t v\}-fg(u\t v,X)u\\
   =&-f\left((u\t v)\t(u\t X)\right)-fv\t X+fg(X,u\t v)u-fX\t v+fg(u,X)(u\t v)-fg(u\t v,X)u\\
   =&-\phi_2\phi_1(X)+\eta_1(X)\xi_2\\
  \end{split}
 \end{equation}
Thus, we have shown

\begin{thm}
 Let $(M,\vp)$ be a closed $7$-manifold with (not necessarily integrable) $G_2$-structure $\vp$. Then $M$ admits an almost contact metric $3$-structure which is compatible with the $G_2$-metric.
\end{thm}

Assume now that the $G_2$-structure is integrable so that $(M,\vp)$ is a $G_2$-manifold. Recall that the third almost contact metric structure induced by the two almost contact metric structures constructed above is given by $\phi_3=\phi_1\phi_2-\eta_2\ot\xi_1$, $\xi_3=\phi_1\xi_2$ and $\eta_3=\eta_1\circ\phi_2$. From the above calculation, we know what $\phi_3$ looks like. It is straightforward to verify that $\xi_3$ is given by
\begin{equation}
 \xi_3=-v+\eta_1(v)u
\end{equation}
Of course, $\eta_3$ is the dual $1$-form to $\xi_3$ using the $G_2$-metric. In particular, we cannot use the results of the previous section on this almost contact metric structure, though they do apply to the almost contact metric structures $(\phi_i,\xi_i,\eta_i)$, $i=1,2$.

\begin{thm}
Let $(M,\vp)$ be a closed $G_2$-manifold. Then $\d\om_1=\d\om_2=0$. if and only if $\xi_1$ and $\xi_2$ are Killing. In this case, the almost contact metric $3$-structure is $3$-cosymplectic. 
\end{thm}

\begin{proof}
The first assertion follows directly from Theorem \ref{cosymplectic}. 

For the second assertion, Theorem \ref{cosymplectic} implies that if $\xi_1$ and $\xi_2$ are Killing, then $(\phi_i,\xi_i,\eta_i,g)$, $i=1,2$, are cosymplectic; thus, it suffices to show that $(\phi_3,\xi_3,\eta_3,g)$ is cosymplectic. Now, note that $\d\om_1=\d\om_2=0$ and $\d\eta_1=\d\eta_2=0$ by definition, that $\na\xi_1=\na\xi_2=0$ since $G_2$-manifolds are Ricci-flat and that $\na\phi_1=\na\phi_2=0$ by Equation \eq{nablaphi}. In fact, it is well-known that an almost contact metric structure is cosymplectic if and only if $\phi$ is parallel (see, e. g. \cite[Theorem $6.8$]{Blair2}). Let $X$, $Y$ be any vector fields on $M$.
\begin{equation}
 \begin{split}
  (\na_X\phi_3)(Y)=&\na_X\phi_3Y-\phi_3(\na_XY)\\
  =&\na_X(\phi_1(\phi_2Y))-\na_X(\eta_2(Y)\xi_1)-\phi_1(\phi_2(\na_XY))+\eta_2(\na_XY)\xi_1\\
  =&(\na_X\phi_1)(\phi_2Y)+\phi_1((\na_X\phi_2)Y)+\phi_1(\phi_2(\na_XY)-((\na_X\eta_2)Y)\xi_1-\eta_2(\na_XY)\xi_1-\eta_2(Y)(\na_X\xi_1)\\
  &-\phi_1(\phi_2(\na_XY))+\eta_2(\na_XY)\xi_1)\\
  =&-((\na_X\eta_2)Y)\xi_1\\
  =&-(\na_X\eta_2Y-\eta_2(\na_XY))\xi_1\\
  =&-(\na_Xg(\xi_2,Y)-g(\xi_2,\na_XY))\xi_1\\
  =&-(g(\na_X\xi_2,Y)+g(\xi_2,\na_XY)-g(\xi_2,\na_XY))\xi_1\\
  =&0\\
 \end{split}
\end{equation}
\end{proof}

\begin{cor}
 On a closed $G_2$-manifold $(M,\vp)$ with full $G_2$-holonomy, the almost contact metric $3$-structure constructed above cannot be an almost cosymplectic $3$-structure and hence cannot be a cosymplectic $3$-structure.
\end{cor}

The proof is similar to the proof of Corollary \ref{G2manifold} and will be omitted. The following result follows directly from Corollary \ref{sasakian}.

\begin{cor}
 On a closed $G_2$-manifold $(M,\vp)$ the almost contact metric $3$-structure constructed above cannot be a contact metric $3$-structure (equivalently, it cannot be $3$-Sasakian).
\end{cor}

\bibliographystyle{amsplain}
\bibliography{ajbibliography}
\end{document}